\documentclass[11pt]{article}

\usepackage{latexsym}
\usepackage{amssymb}
\usepackage{amsthm}
\usepackage{amscd}
\usepackage{amsmath}
\usepackage[all]{xy}
\usepackage{graphicx}

\newtheorem{theorem}{Theorem}[section]

\newtheorem{lemma}{Lemma}[section]

\newtheorem{corollary}{Corollary}[section]

\xyoption{dvips}

\numberwithin{equation}{section}

\newcommand{\FF}{\mathbb{F}}

\newcommand{\CC}{\mathbb{C}}

\def\ZZ{\mathbb{Z}}

\def\Ind{\mathrm{Ind}} 
\def\GL{\mathrm{GL}}

\def\Irr{\mathrm{Irr}}

\newcommand{\cS}{\mathcal{S}}

\newcommand{\U}{\mathrm{U}}
\newcommand{\Sp}{\mathrm{Sp}}
\newcommand{\epi}{\varepsilon_{\iota}}

\def\cP{\mathcal{P}}
\newcommand{\tPhi}{\tilde{\Phi}}
\newcommand{\bmu}{\boldsymbol{\mu}}
\newcommand{\blam}{\boldsymbol{\lambda}}
\newcommand{\bnu}{\boldsymbol{\nu}}
\newcommand{\bgamma}{\boldsymbol{\gamma}}

\makeatletter
\renewcommand{\@makefnmark}{\mbox{\textsuperscript{}}}
\makeatother

\def\adots{\mathinner{\mkern2mu\raise0pt\hbox{.}  
\mkern2mu\raise4pt\hbox{.}\mkern1mu
\raise7pt\vbox{\kern7pt\hbox{.}}\mkern1mu}}

\allowdisplaybreaks[1]

\begin{document}

\bibliographystyle{amsplain}

\title{Values of characters sums for finite unitary groups}
\author{Nathaniel Thiem\footnote{Stanford University: \textsf{thiem@math.stanford.edu}}{ } and C. Ryan Vinroot\footnote{University of Arizona: \textsf{vinroot@math.arizona.edu}}}
\date{}

\maketitle




\begin{abstract}
A known result for the finite general linear group $\GL(n,\FF_q)$ and for the finite unitary group $\U(n,\FF_{q^2})$ posits that the sum of the irreducible character degrees is equal to the number of symmetric matrices in the group.  Fulman and Guralnick extended this result by considering sums of irreducible characters evaluated at an arbitrary conjugacy class of $\GL(n,\FF_q)$.   We develop an explicit formula for the value of the permutation character of $\U(2n,\FF_{q^2})$ over $\Sp(2n,\FF_q)$ evaluated an an arbitrary conjugacy class and use results concerning Gelfand-Graev characters to obtain an analogous formula for $\U(n,\FF_{q^2})$ in the case where $q$ is an odd prime.  These results are also given as probabilistic statements. 
\end{abstract}

\section{Introduction\protect\footnote{MSC 2000: 20C33 (05E05)}\protect\footnote{Keywords: finite unitary group, character sums, conjugacy, Hall-Littlewood functions}}

An important topic of interest in probabilistic group theory is the study of the statistical behavior of random conjugacy classes.   Fulman and Guralnick study this question for the finite general linear group by evaluating several character sums at arbitrary elements \cite{fulgural}. One of the main tools used there is a {\em model} for the group ${\rm GL}(n, \FF_q)$, which is a way of writing the sum of all of the irreducible characters of the group as a sum of characters which are induced from linear characters of subgroups.  The model for ${\rm GL}(n, \FF_q)$ is obtained by a sum of Harish-Chandra induction of Gelfand-Graev characters and permutation characters of the finite symplectic group in various Levi subgroups.  Values of the Gelfand-Graev characters of the finite general linear group are known, and the values of the permutation character of the finite symplectic group are obtained using results of $\GL(2n, \FF_q)$-conjugacy in ${\rm Sp}(2n, \FF_q)$ due to Wall \cite{wall}.  

The purpose of this paper is to make the parallel computations for the finite unitary group $\U(n, \FF_{q^2})$, which involve several structural differences from \cite{fulgural}.  As in \cite{fulgural}, we rely on a model for the finite unitary group, and the values for the permutation character of the finite symplectic group.   The model for the finite unitary group involves replacing Harish-Chandra induction by the more general Deligne-Lusztig induction, and can be found in \cite{TV05}.  To compute the permutation character, we use \cite{wall} and \cite{fulconj} to translate the conjugacy class information from $\Sp(2n, \FF_q)$ to $\U(2n, \FF_{q^2})$ (in the case of $\GL(n,\FF_q)$, the corresponding results are explicitly in \cite{wall}).  The main results obtained are only proven for odd $q$, because the decomposition of the permutation character of $\U(2n, \FF_{q^2})$ on $\Sp(2n, \FF_q)$ is not known for the case that $q$ is even.  An analogous decomposition for the case $q$ even would immediately imply the results here extend to all $q$.

The organization of the paper is as follows.  Section \ref{symmetric} reviews definitions and results on partitions and symmetric functions, with a particular emphasis on the Hall-Littlewood symmetric functions.  Section \ref{conjclasses} describes the conjugacy classes of the finite general linear, unitary, and symplectic groups, and gives the sizes of centralizers in terms of the combinatorial information which parameterize these conjugacy classes.

The main results are in Section \ref{charsums}.  In particular, Theorem \ref{permcharvals} computes the value of the permutation character
$$ {\rm Ind}_{\Sp(2n, \FF_q)}^{\U(2n, \FF_{q^2})} ({\bf 1})$$
at an arbitrary class of ${\rm U}(2n, \FF_{q^2})$, and Theorem \ref{fullsum}  evaluates the sum
$$ \sum_{\chi \in {\rm Irr}(\U(n, \FF_{q^2}))} \chi $$
evaluated at an arbitrary class of ${\rm U}(n, \FF_{q^2})$.  Finally, we translate these main results into probabilistic  statements in Section \ref{probs}.




\section{Partitions and symmetric functions}\label{symmetric}

This section will review some fundamental definitions and results used in this paper, including partitions, Hall polynomials, and symmetric functions. 

\subsection{Partitions}Let 
$$\cP=\bigcup_{n\geq 0} \cP_n,\qquad \text{where} \qquad \cP_n=\{\text{partitions of $n$}\}.$$
For $\nu = (\nu_1, \nu_2, \ldots, \nu_l) \in \cP_n$, the {\em length} $\ell(\nu)$ of $\nu$ is the number of parts $l$ of $\nu$, and the {\em size} $|\nu|$ of $\nu$ is the sum of the parts $n$.  Let $\nu'$ denote the conjugate of the partition $\nu$.    We also write
$$\nu=(1^{m_1(\nu)} 2^{m_2(\nu)} \cdots),\qquad\text{where}\qquad m_i(\nu) = |\{ j\in \ZZ_{\geq 1} \mid \nu_j = i \}|.$$
Let $o(\nu)$ denote the number of odd parts of $\nu$, and define $n(\nu)$ to be
$$n(\nu)=\sum_j (j - 1) \nu_i.$$
The following will be used in calculating a sign in one of the main results.

\begin{lemma} \label{evensign}
Let $\nu \in \cP$ be such that either $m_i(\nu) \in 2\ZZ_{\geq 0}$ whenever $i$ is even, or $m_i(\nu) \in 2\ZZ_{\geq 0}$ whenever $i$ is odd.  Then
$$ (|\nu| - o(\nu))/2 \equiv \lfloor |\nu|/2 \rfloor + n(\nu) \; (\text{mod } 2).$$
\end{lemma}
\begin{proof}  An alternate statement of the Lemma is if $\nu$ satisfies one of the conditions above and $|\nu|$ is even, then $n(\nu) + o(\nu)/2$ is even, and if $\nu$ satisfies the first condition above and $|\nu|$ is odd, then $n(\nu) + (o(\nu) - 1)/2$ is even. 

First suppose that $m_i(\nu)$ is even whenever $i$ is odd, so that $|\nu|$ and $o(\nu)$ are both even.  This implies that exactly half of the odd parts of $\nu$ are of the form $\nu_{2j}$.  Consequently,
$$ n(\nu) + o(\nu)/2 = \sum_j (j-1)\nu_j + o(\nu)/2 $$
is even, giving the result.

Now suppose that $m_i(\nu)$ is even whenever $i$ is even, and further suppose that $|\nu|$ is even, so that $o(\nu)$ is even as well.  Then exactly half of the even parts of $\nu$ are of the form $\nu_{2j}$, which implies that exactly half of the odd parts of $\nu$ are of this form as well.  Again we have that $n(\nu) + o(\nu)/2$ is even.

Finally, suppose that $m_i(\nu)$ is even whenever $i$ is even, but that $|\nu|$ is odd, so that $o(\nu)$ is odd.  This implies that $\ell = \ell(\nu)$ is also odd.  Now choose an odd number $k$ such that $k \leq \nu_{\ell}$, and define a new partition $\tilde{\nu}$ by
$$ \tilde{\nu} = (\nu_1, \nu_2, \ldots, \nu_{\ell}, k).$$
Now, $\tilde{\nu}$ satisfies the condition that $m_i(\tilde{\nu})$ is even whenever $i$ is even, except now we have that $|\tilde{\nu}|$ is even and $o(\tilde{\nu}) = o(\nu) + 1$ is even as well.  From the previous case, we have that $n(\tilde{\nu}) + o(\tilde{\nu})/2$ is even.  We also have
$$n(\tilde{\nu}) - n(\nu) = \ell k \;\; \text{ and } \;\; o(\tilde{\nu})/2 - (o(\nu) - 1)/2 = 1$$
are odd.  Therefore $n(\nu) + (o(\nu) - 1)/2$ is even.
\end{proof}

\subsection{Hall polynomials} Let $R$ be a discrete valuation ring with maximal ideal $\wp$, and with finite residue field of size $q$.  Finite $R$-modules are then parameterized by partitions, where the module $M$ of type $\lambda \in \cP$ is isomorphic to
$$ \bigoplus_{i=1}^{\ell(\lambda)} R/ \wp^{\lambda_i}.$$
Given a finite $R$-module $M$ of type $\lambda$, the number of submodules $N$ of type $\mu$ such that $M/N$ is of type $\nu$ is a polynomial in $q$ (see \cite[Chapter II]{Mac}), and is the called the \emph{Hall polynomial}, written $g_{\mu \nu}^{\lambda}(q)$.  We may thus consider the Hall polynomial in some indeterminate $t$, $g_{\mu\nu}^{\lambda}(t)$.  Note that we have
\begin{equation*}
g_{\mu\nu}^{\lambda} (t) = 0 \;\; \text{ unless } \;\; |\mu|+|\nu| = |\lambda| \;\; \text{ and } \;\; \mu, \nu \subset \lambda.
\end{equation*}
In particular, we have
\begin{equation} \label{hallempty}
g_{\emptyset \nu}^{\lambda}(t) = \begin{cases} 1 & \text{if } \nu=\lambda, \\ 0 & \text{otherwise,} \end{cases}
\end{equation}
where $\emptyset$ denotes the empty partition.

\subsection{The ring of symmetric functions}  A \emph{symmetric polynomial} $f \in \ZZ[x_1, x_2, \ldots, x_n]$ is a polynomial which is invariant under the action of the symmetric group $S_n$ permuting the variables.  Let $\Lambda_n$ be the set of symmetric polynomials in $\ZZ[x_1, x_2, \ldots, x_n]$, so that 
$$ \Lambda_n = \ZZ[x_1, x_2, \ldots, x_n]^{S_n}.$$
Following \cite[I.2]{Mac}, let $\Lambda_n^k$ denote the homogeneous elements in $\Lambda_n$ of degree $k$, and for $m>n$, we have the natural projection map 
$$ p_{m,n}^k: \Lambda_m^k \rightarrow \Lambda_n^k,$$
which may be used to form the inverse limit
$$ \Lambda^k = \lim_{\longleftarrow} \Lambda^k_n.$$
We then define $\Lambda$, the \emph{ring of symmetric functions} over $\ZZ$ in the countably infinite list $x = \{x_1, x_2, \ldots \}$ of independent variables, to be
$$ \bigoplus_k \Lambda^k,$$
and for $t \in \CC^{\times}$, let
$$\Lambda[t]=\ZZ[t]\otimes_\ZZ \Lambda.$$  
\subsection{Hall-Littlewood symmetric functions}  A symmetric function over $\ZZ[t]$ of central importance here is the \emph{Hall-Littlewood} symmetric function $P_{\lambda}(x;t)$, where $\lambda \in \cP$.  A thorough discussion of Hall-Littlewood functions may be found in \cite[Chapter III]{Mac}, but we give the definition here for completeness.

Let $\lambda \in \cP$ such that $\ell(\lambda) \leq n$, where we let $\lambda_i = 0$ if $i > \ell(\lambda)$.  For any $\sigma \in S_n$ and polynomial $f$ in $x_1, \ldots, x_n$, let $\sigma(f)$ denote the action of $\sigma$ on $f$ by permuting the variables.  The Hall-Littlewood symmetric polynomial is defined to be
$$ P_{\lambda}(x_1, \ldots, x_n; t) = \prod_{i \geq 0} \prod_{j=1}^{m_i(\lambda)} \frac{1-t}{1-t^j} \sum_{\sigma \in S_n} \sigma \Big(x_1^{\lambda_1} \cdots x_n^{\lambda_n} \prod_{i <j} \frac{x_i - tx_j}{x_i - x_j} \Big).$$
Then the Hall-Littlewood symmetric function $P_{\lambda}(x ; t)$ is obtained by finding the image of the Hall-Littlewood symmetric polynomials through the limiting process described above.  The set of functions $P_{\lambda}(x ; t)$, as $\lambda$ ranges over all partitions, forms a $\ZZ[t]$-basis of $\Lambda[t]$ (\cite[III.2.7]{Mac}).  

Hall polynomials show up as coefficients when products of Hall-Littlewood functions are expressed as a sum of Hall-Littlewood functions, as given in the following result, found in \cite[III.3]{Mac}.

\begin{lemma} \label{HLprod}
$$t^{n(\mu)} P_{\mu}(x ; t) t^{n(\nu)} P_{\nu}(x;t) = \sum_{\lambda \in \cP} g_{\mu \nu}^{\lambda} (1/t) t^{n(\lambda)} P_{\lambda}(x ; t).$$
\end{lemma}

We will use the following identity, which is obtained in \cite[III.3, Example 1]{Mac}.

\begin{lemma} \label{HLsum1}
$$ \sum_{\lambda \in \cP} t^{n(\lambda)} \prod_{j=1}^{\ell(\lambda)} (1 + t^{1-j} y) P_{\lambda}(x ; t) = \prod_{j \geq 1} \frac{1 + x_j y}{1 - x_j}.$$
\end{lemma}

We may also view a partition $\lambda \in \cP$ as a set of ordered pairs $(i,j)$ of positive integers, where $1 \leq j \leq \lambda_i$ and $1 \leq i \leq \lambda'_j$.  For a parameter $t$, we define the function $c_{\lambda}(t)$ as 
\begin{equation} \label{cdefn} c_{\lambda}(t) = \prod_{ (i,j) \in \lambda, \lambda_i = j \atop{ \lambda_j' - i \text{ even }}} ( 1 - t^{\lambda_j' - i + 1}).
\end{equation}
For the pair $(i,j) \in \lambda$, the quantity $\lambda_i - j$ is called the {\em arm} of $(i,j)$, and $\lambda_j' - i$ is called the {\em leg} of $(i,j)$.  Note that we also have
\begin{equation} \label{cprod}
c_{\lambda}(t) = \prod_{i} \prod_{j \text{ odd} \atop{j \leq m_i(\lambda)}} ( 1 - t^j).
\end{equation}
  The following identity involving Hall-Littlewood functions and the function $c_{\lambda}(t)$ was proven by Kawanaka in \cite{kawanaka}, and a purely combinatorial proof was later given by Fulman and Guralnick in \cite{fulgural}.

\begin{theorem} [Kawanaka] \label{HLKa}
$$ \sum_{\lambda \in \cP \atop{ m_k(\lambda) \text{ even } \atop{ \text{for $k$ odd}} }} t^{(o(\lambda) - |\lambda|)/2}c_{\lambda}(t) P_{\lambda}(x;t) = \prod_{i \leq j} \frac{1-x_ix_j}{1 - x_i x_j/t}. $$
\end{theorem}

The next identity, also involving Hall-Littlewood functions and $c_{\lambda}(t)$, is due to Fulman and Guralnick \cite[Theorem 2.8]{fulgural}, and is crucial to obtaining our main result.

\begin{theorem} [Fulman, Guralnick] \label{HLFG}
$$ \sum_{\lambda \in \cP \atop{ m_k(\lambda) \text{ even} \atop{ \text{for $k$ even}} }} \frac{c_{\lambda}(t) P_{\lambda}(x;t)}{t^{(o(\lambda) + |\lambda|)/2}} = \prod_{i \geq 1} \frac{1 + x_i/t}{1+x_i} \prod_{i \leq j} \frac{1- x_i x_j}{1 - x_i x_j/t}. $$
\end{theorem}

\section{The groups and their conjugacy classes} \label{conjclasses}

After reviewing the definitions of the finite classical groups and their orders, this section analyzes how their conjugacy classes interact and computes the corresponding centralizer subgroup sizes.  

\subsection{The finite classical groups}

Let $\bar{G}_n=\GL(n,\bar\FF_q)$ be the general linear group with entries in the algebraic closure of the finite field $\FF_q$ with $q$ elements.    Let 
$$\begin{array}{rccc}F:& \bar{G}_n & \longrightarrow & \bar{G}_n\\ & (a_{ij}) & \mapsto & (a_{ij}^q)\end{array}$$
be the usual Frobenius homomorphism.  Then the finite general linear groups are given by 
$$\GL(n,\FF_{q^k})=\mathrm{Stab}_{\bar{G}_n}(F^k).$$
The finite unitary groups are given by
\begin{equation*}
\mathrm{U}(n,\FF_{q^{2}})=\{(a_{ij})\in \bar{G}_{n}\ \mid\ (a_{ji}^{q})^{-1}=(a_{ij})\}.
\end{equation*}
Let 
$$J=\left(\begin{array}{ccc|ccc} 
& & & 0 & & -1\\ 
& 0 & & & \adots & \\ 
& & & -1 & & 0 \\ \hline
0 & & 1 & & & \\
& \adots & & & 0 & \\
1 & & 0 & & & \end{array}\right),$$
and let $w=(w_{ij})\in \bar{G}_n$ be an element such that $(w_{ji}^q)^{-1}(w_{ij})^{-1}=J$, whose existence is guaranteed by the Lang-Steinberg Theorem. Then
the finite symplectic groups are given by
\begin{align*}
\mathrm{Sp}(2n,\FF_{q}) & =\{(a_{ij})\in \GL(2n,\FF_{q})\ \mid\ (a_{ij})J(a_{ji})=J\}\\
&\cong\{(a_{ij})\in \mathrm{U}(2n,\FF_{q^{2}})\ \mid\  (a_{ij})w^{-1}Jw (a_{ji})=w^{-1}Jw, wgw^{-1}\in \GL(2n,\FF_{q})\}.
\end{align*}
Using the second definition above, we see that ${\rm Sp}(2n, \FF_q)$ may be viewed as a subgroup of ${\rm U}(2n, \FF_{q^2})$.

We will only deal with finite orthogonal groups (see, for example, \cite{grove} for a definition) in the case that $q$ is odd.  When $m$ is even, ${\rm O}^+(m, \FF_q)$ and ${\rm O}^-(m, \FF_q)$ are the split and non-split finite orthogonal groups, respectively, and when $m$ is odd, ${\rm O}(m, \FF_q)$ is the unique finite orthogonal group.

The orders of the groups of interest in this paper are
\begin{align*}
|\GL(n,\FF_{q})| & = q^{n(n-1)/2}(q^{n}-1)(q^{(n-1)}-1)\cdots (q-1)\\
|\mathrm{U}(n,\FF_{q^{2}})| &= q^{n(n-1)/2}(q^{n}-(-1)^{n})(q^{(n-1)}-(-1)^{(n-1)})\cdots (q+1))\\
|\mathrm{Sp}(2n,\FF_{q})| &=q^{n^2}(q^{2n}-1)(q^{2(n-1)}-1)\cdots (q^2-1)\\
|\mathrm{O}^{\pm}(2n,\FF_{q})| &= 2q^{n^2}(1\pm q^{-n})(q^{2(n-1)}-1)(q^{2(n-2)}-1)\cdots (q^2-1)\\
|\mathrm{O}(2n+1,\FF_{q})|&=2q^{n^2}(q^{2n}-1)(q^{2(n-1)}-1)\cdots (q^2-1).
\end{align*}  

\subsection{Conjugacy classes}

This section examines the conjugacy classes of the following groups, and the behavior of conjugacy classes when considering one group as it is contained in another:
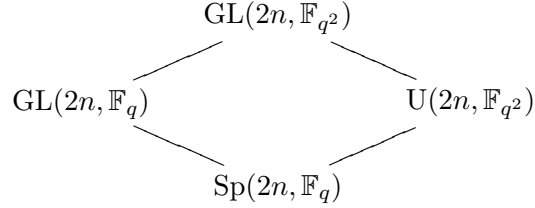
\begin{equation*}\xy<0cm,1.5cm>\xymatrix@R=.5cm@C=.5cm{ & {\GL(2n,\FF_{q^2})} \\ {\GL(2n,\FF_q)} \ar @{-} [ur] & & {\U(2n,\FF_{q^2})}\ar @{-} [ul]  \\ & {\Sp(2n,\FF_q)}\ar @{-} [ul] \ar @{-} [ur] }\endxy\end{equation*}
Define sets,
\begin{equation*}
\Phi =\Phi_1 \qquad\text{where for $k\in\ZZ_{\geq 1}$,}\qquad \Phi_k  =\{F^k\text{-orbits in } \bar{G}_1\}, \end{equation*}
and
\begin{align*}
\tilde{\Phi}_2 &=  \{\sim\text{-orbits in } \Phi_2\}, & \text{where}\quad & \begin{array}{rccc} \sim: & \Phi_2 & \longrightarrow & \Phi_2\\ 
 & s=\{s_1,\ldots, s_k\} &\mapsto & \tilde{s}=\{s_1^{-q},\ldots, s_k^{-q}\}\end{array}\\ 
\Phi^* &= \{\ast\text{-orbits in } \Phi\}, &  \text{where}\quad & \begin{array}{rccc} \ast: & \Phi & \longrightarrow & \Phi\\ 
 & s=\{s_1,\ldots, s_k\} &\mapsto & s^\ast=\{s_1^{-1},\ldots, s_k^{-1}\}.\end{array}
\end{align*}
Note that there is an injective map
$$\begin{array}{ccc}\left\{\text{Subsets of } \bar{G}_1 \right\} & \longrightarrow & \bar{\FF}_q[X]\\ s=\{s_1,s_2,\ldots, s_k\} & \mapsto & f_s=(X-s_1)(X-s_2)\cdots(X-s_k),\end{array}$$ 
that sends $\Phi_k$ to the $\FF_{q^k}$-irreducible polynomials in $\FF_{q^k}[X]$.  Thus, we can identify the sets of orbits $\Phi_k$, $\tilde\Phi_2$ and $\Phi^\ast$ with sets of polynomials in the variable $X$.  Also, for $s = \{s_1, \ldots, s_k \}$, we write $s^q = \{s_1^q, \ldots, s_k^q \}$.

A \emph{$\Phi$-partition} $\blam: \Phi \rightarrow \cP$ is a function which assigns a partition $\blam^{(s)}$ to each orbit $s \in \Phi$.  We may also think of a $\Phi$-partition as a sequence of partitions indexed by $\Phi$.  The \emph{size} of $\blam$ is
$$|\blam|=\sum_{s\in \Phi} |s||\blam^{(s)}|,$$
where $|s|$ is the size of the orbit $s$.  Let
$$\cP^\Phi=\bigcup_{n\geq 0} \cP_n^\Phi,\qquad\text{where}\qquad \cP_n^\Phi=\{\text{$\Phi$-partitions of size $n$}\}.$$
We can define $\Phi_k$-partitions, $\tilde\Phi_2$-partitions, and $\Phi^*$-partitions similarly.

By Jordan rational form (see \cite[IV.2]{Mac}),
$$\left\{\begin{array}{c} \text{Conjugacy classes}\\ \text{of $\GL(n,\FF_{q^k})$}\end{array}\right\} \longleftrightarrow \cP^{\Phi_k}_n$$
such that the conjugacy class corresponding to $\blam\in \cP^{\Phi_k}$ has characteristic polynomial 
$$\prod_{s\in \Phi_k} f_s^{|\blam^{(s)}|}.$$
Furthermore, Wall \cite{wall} and Ennola \cite{ennolaconj} established
$$\left\{\begin{array}{c} \text{Conjugacy classes}\\ \text{of $\U(n,\FF_{q^2})$}\end{array}\right\} \longleftrightarrow \cP^{\tilde\Phi_2}_{n}$$
such that the conjugacy class corresponding to $\blam\in  \cP^{\tilde\Phi_2}$ has characteristic polynomial 
$$\prod_{s\in  \tilde\Phi_2} f_s^{|\blam^{(s)}|}.$$

The following Lemma addresses the relationship between the conjugacy classes of these groups.  By a slight abuse of notation, we will write
$$1=\{1\}\in \Phi\qquad \text{and}\qquad -1=\{-1\}\in \Phi.$$
So, $s = 1$ corresponds to the polynomial $X-1$, and $s = -1$ to $X+1$.

\begin{lemma}\label{UpDownConjugacy} \hfill

\begin{enumerate}
\item[(a)] The element $g_{\blam}\in \GL(n,\FF_{q^2})$, in the conjugacy class given by $\blam \in \cP^{\Phi_2}_n$, is conjugate to an element of $\GL(n,\FF_q)$ if and only if the $\Phi$-partition $\bmu$ given by
$$\bmu^{(s\cup s^q)}=\blam^{(s)}, \qquad\text{for $(s\cup s^q)\in \Phi$,}$$
is well-defined.  
\item[(b)] The element $g_{\blam}\in \GL(n,\FF_{q^2})$, in the conjugacy class given by $\blam \in \cP^{\Phi_2}_n$, is conjugate to an element in $\mathrm{U}(n,\FF_{q^2})$ if and only if the $\tilde{\Phi}$-partition $\bmu$ given by
$$\bmu^{(s\cup\tilde{s})}=\blam^{(s)}, \qquad\text{for $(s\cup \tilde{s})\in \tilde\Phi_2$,}$$
is well-defined.
\item[(c)] Let $q$ be odd.  The element $g_{\blam}\in \GL(2n,\FF_{q})$, in the conjugacy class given by $\blam \in \cP^{\Phi}_{2n}$, is conjugate to an element in $\mathrm{Sp}(2n,\FF_{q})$ if and only if the $\Phi^*$-partition $\bmu$ given by
$$\bmu^{(s\cup s^*)}=\blam^{(s)}, \qquad\text{for $(s\cup s^*)\in \Phi^*$,}$$
is well-defined and $m_j(\blam^{(1)}),m_j(\blam^{(-1)})\in 2\ZZ_{\geq 0}$ for all odd $j\in \ZZ_{\geq 1}$.
\item[(d)] Let $q$ be odd.  The element $g_{\blam}\in \mathrm{U}(2n,\FF_{q^2})$, in the conjugacy class given by $\blam \in \cP^{\tPhi_2}_{2n}$, is conjugate to an element in $\mathrm{Sp}(2n,\FF_{q})$ if and only if the $\Phi^*$-partition $\bmu$ given by
$$\bmu^{(s\cup s^*)}=\blam^{(s)}, \qquad\text{for $(s\cup s^*)\in \Phi^*$,}$$
is well-defined and $m_j(\blam^{(1)}),m_j(\blam^{(-1)})\in 2\ZZ_{\geq 0}$ for all odd $j\in \ZZ_{\geq 1}$.
\end{enumerate}
\end{lemma}
\begin{proof}
Part (a) follows from considering Jordan rational forms, or elementary divisors, as in \cite[IV.2]{Mac}, while parts (b) and (c) are due to Wall \cite{wall}. 

(d) Suppose $g_{\blam}\in \mathrm{U}(2n,\FF_{q^2})$, $\bmu$ is well-defined, and and $m_j(\blam^{(1)}),m_j(\blam^{(-1)})\in 2\ZZ_{\geq 0}$ for all odd $j\in \ZZ_{\geq 1}$.  Since $g_{\blam}\in \GL(2n,\FF_{q^2})$, (b) implies there exists $\bnu\in \cP^{\Phi_2}_{2n}$ such that
$$\blam^{(s)}=\bnu^{(r)}, \qquad\text{for $s=(r\cup \tilde{r})\in \tilde\Phi_2$.}$$
However, by assumption there is a $\Phi$-partition $\bgamma$ such that
$$\bgamma^{(r\cup \tilde{r}^*=r^q)}=\bnu^{(r)},\qquad\text{for $(r \cup r^q)\in \Phi$,}$$
is well-defined.  Thus, $g_{\blam}$ is conjugate to an element in $\GL(2n,\FF_q)$ and the result follows from (c).

Suppose $g_{\blam}\in \mathrm{U}(2n,\FF_{q^2})$ is conjugate to an element in $\Sp(2n,\FF_q)$.  Since $\Sp(2n,\FF_q)\subseteq \GL(2n,\FF_q)$, we have that $g_{\blam}$ is conjugate to an element of $\GL(2n,\FF_q)$, and the result follows from (c).
\end{proof}

We note that in each of these groups, the corresponding conjugacy class is unipotent exactly when $\blam^{(1)}$ is the only nonempty partition of the $\Phi_k$-partition or $\tPhi_2$-partition.

For the rest of this section, we let $q$ be odd.  The conjugacy classes of $\Sp(2n,\FF_q)$ are not parameterized by $\Phi^*$-partitions $\blam$ of size $2n$ such that $m_j(\blam^{(1)}),m_j(\blam^{(-1)})\in 2\ZZ_{\geq 0}$ for all odd $j\in \ZZ_{\geq 1}$, as might be suggested by Lemma \ref{UpDownConjugacy}.  Instead, we need to distinguish between some classes in $\Sp(2n,\FF_q)$ that may be conjugate in $\GL(2n,\FF_q)$.  The development we give here follows \cite{fulconj}.

 A \emph{symplectic signed partition} is a partition $\lambda$ with a function
$$\delta:\{2i\in 2\ZZ_{\geq 1}\ \mid\ m_{2i}(\lambda)\neq 0\}\longrightarrow \{\pm 1\}$$
such that $m_j(\lambda)\in 2\ZZ_{\geq 0}$ for all odd $j\in \ZZ_{\geq 1}$.

For example,
$$\gamma = (5,5,4^{-},4^-,3,3,3,3,2^{+}, 2^+, 2^+, 1, 1) = (5^2, 4^{-2}, 3^4, 2^{+3}, 1^2),$$
is a symplectic signed partition, where the parts of size $4$ are assigned the sign $-$ and the parts of size $2$ the sign $+$.  For even $i$, the multiplicity of $i$ in $\gamma$, $m_i(\gamma)$, will be given the sign that the parts $i$ are assigned.  In the example above, we have $m_4(\gamma) = -2$, ${\rm sgn}(m_4(\gamma)) = -$, $m_2(\gamma) = +2$, and ${\rm sgn}(m_2(\gamma)) = +$.  Let $\cP^\pm$ denote the set of all symplectic signed partitions.

If $\gamma$ is a symplectic signed partition, we let $\gamma^{\circ}$ denote the partition obtained by ignoring the signs of the sets of even parts of $\gamma$.  We define various functions on symplectic signed partitions as its value on $\gamma^{\circ}$, for example, we define $|\gamma|$ and $n(\gamma)$ as $|\gamma^{\circ}|$ and $n(\gamma^{\circ})$, respectively.  Define $\cP^\pm_n$ as the set of symplectic signed partitions of size $n$.

Let $\cP^{\pm\Phi^\ast}$ denote the set of functions $\blam: \Phi^* \rightarrow \cP \cup \cP^\pm$ such that
\begin{enumerate}
\item[(a)] $\blam^{(1)},\blam^{(-1)}\in \cP^\pm$ are symplectic signed partitions,
\item[(b)] $\blam^{(r)}\in \cP$ for $r \in \Phi^*$, $r \neq \pm 1$.
\end{enumerate}
For $\blam \in \cP^{\pm\Phi^\ast}$, define
$$|\blam| = \sum_{r \in \Phi^*} |r|\,|\blam^{(r)^{\circ}}|,$$
and let 
$$\cP^{\pm\Phi^\ast}_{2n} = \{ \blam \in \cP^{\pm\Phi^\ast} \;\; | \;\; |\blam| = 2n \}.$$

Then from results of Wall \cite{wall},
$$\left\{\begin{array}{c} \text{Conjugacy classes}\\ \text{of $\Sp(2n,\FF_{q})$}\end{array}\right\} \longleftrightarrow \cP^{\pm\Phi^*}_{2n}$$
such that the conjugacy class corresponding to $\blam\in  \cP^{\pm\Phi^*}$ has characteristic polynomial 
$$\prod_{r\in  \Phi^*} f_r^{|\blam^{(r)}|}.$$

Alternatively, we may consider the set $\cP^{\pm\Phi}$ of functions $\bgamma: \Phi \rightarrow \cP \cup \cP^\pm$, such that
\begin{enumerate}
\item[(a)] $\bgamma^{(1)},\bgamma^{(-1)}\in \cP^\pm$ are symplectic signed partitions,
\item[(b)] $\bgamma^{(r)} = \bgamma^{(r^*)}$ for all $r \in \Phi$,
\item[(c)] $\bgamma^{(r)} \in \cP$ for $r \in \Phi$, $r \neq \pm 1$.
\end{enumerate}
If we define $|\bgamma|$ and $\cP^{\pm\Phi}_{2n}$ analogously, then the set $\cP^{\pm\Phi}_{2n}$ also parameterizes the conjugacy classes of ${\rm Sp}(2n, \FF_q)$.  In this case, if $g$ is an element of ${\rm Sp}(2n, \FF_q)$ in the conjugacy class corresponding to $\bgamma \in \cP^{\pm\Phi}_{2n}$, then the conjugacy class of $g$ in ${\rm GL}(2n, \FF_q)$ is given by $\blam \in \cP^{\Phi}_{2n}$, where $\blam^{(r)} = \bgamma^{(r)^\circ}$ for $r = \pm 1$, and $\blam^{(r)} = \bgamma^{(r)}$ otherwise.

\subsection{Centralizers}

For $g$ in a group $G$, let
$$C_G(g)=\{h\in G\ \mid\ gh=hg\},$$
denote the centralizer of $g$ in $G$. 

\begin{theorem}[Ennola, Wall] \label{wall} 
Let $g_{\bmu} \in \U(n,\FF_{q^2})$ be in the conjugacy class indexed by $\bmu \in \cP^{\tilde\Phi_2}_{n}$.  The order of the centralizer of $g_{\bmu}$ in $\U(n,\FF_{q^2})$ is 
$$a_{\bmu}=(-1)^{|\bmu|}\prod_{s\in \tilde\Phi_2}a_{\bmu^{(s)}} \bigl((-q)^{|s|}\bigr),\text{ where }
a_{\mu}(x) = x^{|\mu| + 2n(\mu)} \prod_i \prod_{j = 1}^{m_i} (1 - (1/x)^j),$$
for $\mu=(1^{m_1}2^{m_2}3^{m_3}\cdots)\in \cP$. 
\end{theorem}

Let $q$ be odd, and let $\bgamma \in \cP^{\pm\Phi}$.  For each $r \in \Phi$, define $A^*(r, \bgamma^{(r)}, i)$ as follows, where we let $m_i = m_i(\bgamma^{(r)})$.
\begin{equation} \label{Atildedefn}
A^*(r, \bgamma^{(r)}, i) = \left\{\begin{array}{ll}\big|{\rm U}(m_i, \FF_{q^{|r|}})\big| &\text{if $r = r^*$, $r \neq \pm1$,}\\ \big|{\rm GL}(m_i, \FF_{q^{|r|}})\big|^{1/2} & \text{if $r \neq r^*$, $r \neq \pm1$,} \\ \big|{\rm Sp}(m_i, \FF_q)\big| & \text{if $r = \pm1$, $i$ odd,} \\ q^{|m_i|/2} \big|{\rm O}^{{\rm sgn}(m_i)}(|m_i|, \FF_q)\big| & \text{if $r = \pm1$, $i$ even.}\end{array}\right.
\end{equation}
In the case that $r=\pm 1$ and $i$ and $m_i$ are both even, ${\rm sgn}(m_i)$ says whether to take the split or non-split orthogonal group, and when $m_i$ is odd, this will always give the unique orthogonal group regardless of the sign of $m_i$.  For each $r \in \Phi$, we define $B^*(r, \bgamma^{(r)})$ as
\begin{equation} \label{Btildedefn}
B^*(r, \bgamma^{(r)}) = q^{|r| ( |\bgamma^{(r)}|/2 + n(\bgamma^{(r)}) - \sum_i m_i^2/2 )} \prod_i A^*(r, \bgamma^{(r)}, i).
\end{equation}
These factors are defined in order to state the following result of Wall \cite{wall}, which tells us the size of a given conjugacy class of $\Sp(2n,\FF_q)$.

\begin{theorem} [Wall] \label{spcent}
Let $g \in \Sp(2n,\FF_q)$ be in the conjugacy class parameterized by $\bgamma \in \cP^{\pm\Phi}_{2n}$.  Then the order of the centralizer of $g$ in $\Sp(2n,\FF_q)$ is given by
$$ |C_{\Sp(2n,\FF_q)}(g)| = \prod_{r \in \Phi} B^*(r, \bgamma^{(r)}).$$
\end{theorem}

\noindent
{\bf Remark.  } We have applied the following identity for partitions in order to state the results of Wall in the form given in Theorem \ref{spcent}.  For any $\gamma \in \cP$, we have
$$ \sum_{i < j} im_i(\gamma)m_j(\gamma) + \frac{1}{2} \sum_i (i-1)m_i(\gamma)^2 = |\gamma|/2 + n(\gamma) - \sum_i m_i(\gamma)^2/2.$$

Note that an element $s \in \tilde{\Phi}_2$ satisfies either $s\in \Phi_2$, or $s=v\cup \tilde{v}$ where $v,\tilde{v}\in \Phi_2$ with $v\neq \tilde{v}$.  For $\bnu \in \cP^{\tilde\Phi_2}$, and $s \in \tilde\Phi_2\setminus\{\pm 1\}$, let 
\begin{equation} \notag
A(s, \bnu^{(s)}, i) = \left\{\begin{array}{ll}\big|{\rm U}(m_i, \FF_{q^{2|s|}})\big|^{1/2} &\text{if $|s|$ is odd,}\\ \big|{\rm GL}(m_i, \FF_{q^{|s|}})\big|^{1/2} & \text{if $|s|$ is even, $s \neq s^*$,} \\ \big|{\rm U}(m_i, \FF_{q^{|s|}})\big| & \text{if $s = v \cup\tilde{v}$, $v= v^*$,} \\ \big|{\rm GL}(m_i, \FF_{q^{|s|/2}})\big| & \text{if $s = v \cup \tilde{v}$, $v = \tilde{v}^*$,}\end{array}\right.
\end{equation}
\begin{equation} \label{ABdefn}
B(s, \bnu^{(s)}) = q^{|s| ( |\bnu^{(s)}|/2 + n(\bnu^{(s)}) - \sum_i m_i^2/2 )} \prod_i A(s, \bnu^{(s)}, i),
\end{equation}
where $m_i = m_i(\bnu^{(s)})$. 

\begin{theorem} \label{spucent}
Let $g \in \Sp(2n,\FF_q)$ be in the $\U(2n,\FF_{q^2})$-conjugacy class parameterized by $\bnu \in \cP^{\tilde\Phi_2}_{2n}$, and in the $\Sp(2n, \FF_q)$-conjugacy class parameterized by $\bgamma \in \cP^{\pm\Phi}_{2n}$.  Then the order of the centralizer of $g$ in $\Sp(2n,\FF_q)$ is given by
$$|C_{\Sp(2n,\FF_q)}(g)| = B^*(1, \bgamma^{(1)}) B^*(-1, \bgamma^{(-1)}) \prod_{ s \in \tilde\Phi_2\setminus\{\pm 1\}} B(s, \bnu^{(s)}).$$
\end{theorem}
\begin{proof} From Theorem \ref{spcent}, we just need to prove that
\begin{equation} \label{resprod}
\prod_{ r \in \Phi\setminus\{\pm 1\}} B^*(r, \bgamma^{(r)}) = \prod_{ s\in \tilde\Phi_2\setminus\{\pm1\}}  B(s, \bnu^{(s)}).
\end{equation}

Suppose $s\in \tilde\Phi_2\setminus\{\pm 1\}$ satisfies $s\in \Phi_2$.  Note that since $|s|$ is odd, the relation $s^*=s$ would imply the existence of some $s_i \in s$ such that $s_i^{-1}=s_i$.  However, $\pm 1\notin s$, so such an $s_i$ cannot exist.  Thus, $s^*\neq s$.  Let $r=s\cup s^*$ so that $r\in \Phi$, $r^*=r$ and $|r|=2|s|$.  Then
$\bgamma^{(r)}=\bnu^{(s)}=\bnu^{(s^*)}$ and by definition
$$B^*(r, \bgamma^{(r)}) = B(s, \bnu^{(s)})B(s^*, \bnu^{(s^*)}).$$

Suppose that $s = v \cup\tilde{v}$ with $v\neq \tilde{v}\in \Phi_2$, and that $s \neq s^*$.  Let $r = v\cup \tilde{v}^*$, so that $r^* = v^* \cup\tilde{v} \neq r$.  Then we have $r \in \Phi$, $|r| = |s|$, $\bnu^{(s)} = \bgamma^{(r)}$, and
$$ B^*(r, \bgamma^{(r)}) = B(s, \bnu^{(s)}) \;\; \text{ and } \;\; B^*(r^*, \bgamma^{(r)}) = B(s^*, \bnu^{(s^*)}).$$

Finally, suppose that $s=v\cup\tilde{v}$ with $v\neq \tilde{v}\in \Phi_2$ and $s = s^*$.  Either $v = v^*$ or $v = \tilde{v}^*$.  If $v = v^*$, then in fact we have $\tilde{s}^* = s$, so that $s\in \Phi$.  Letting $r = s$, we have $\bgamma^{(r)} = \bnu^{(s)}$, and
$$ B^*(r, \bgamma^{(r)}) = B(s, \bnu^{(s)}).$$
If $v = \tilde{v}^*$, then we have $v \in \Phi$, and $v \neq v^*$.  Letting $r = v$, we have $|r| = |s|/2$, and $\bnu^{(s)} = \bgamma^{(r)} = \bgamma^{(r^*)}$, so that
$$ B^*(r, \bgamma^{(r)}) B^*(r^*, \bgamma^{(r^*)}) = B(s, \bnu^{(s)}).$$
This exhausts all factors in both sides of (\ref{resprod}), and the desired result is obtained. 
\end{proof}

\section{Character sums} \label{charsums}

This section contains our main results, including formulas for
$$\Ind_{\Sp(2n,\FF_q)}^{\U(2n,\FF_{q^2})}(\mathbf{1})(g),\quad\text{for $g \in \U(2n, \FF_{q^2})$,}\quad \text{and}\quad \sum_{\chi\in \Irr(\U(n,\FF_{q^2}))} \chi(g), \quad \text{for $g \in \U(n, \FF_{q^2})$},$$
in Theorem \ref{permcharvals} and Theorem \ref{charvalues}, respectively.

\subsection{Deligne-Lusztig induction and a model for $\U(n,\FF_{q^2})$}

For any class function $\chi$ of $\U(n, \FF_{q^2})$, and any conjugacy class of $\U(n, \FF_{q^2})$ corresponding to $\bnu \in \cP^{\tPhi_2}_{n}$, we let $\chi(\bnu)$ denote the value of $\chi$ evaluated on any element of the conjugacy class corresponding to $\bnu$. 

Let $\Gamma_{(n)}$ denote the Gelfand-Graev character of $\U(n,\FF_{q^2})$, which by \cite{TV06} takes values given by the character formula
\begin{equation}\label{UGGFormula}
\Gamma_{(n)}(\bmu)=\left\{\begin{array}{ll} (-1)^{\lfloor n/2 \rfloor} \prod_{i = 1}^{\ell(\mu)} \bigl( 1-(-q)^i \bigr) & \text{if $\bmu$ is unipotent and $\mu=\bmu^{(1)}$,}\\ 0 & \text{otherwise.}\end{array}\right.
\end{equation}

Let
$$C=\bigoplus_{n\geq 0} C_n, \qquad \text{where}\qquad C_n=\{\text{class functions of $\U(n,\FF_{q^2})$}\},$$
be the ring of class functions with multiplication given by Deligne-Lusztig induction
\begin{equation} \label{dlprod}
\chi \circ \eta = R_{U_m \oplus U_n}^{U_{m+n}} (\chi \otimes \xi).
\end{equation}
where  $\chi \in C_m$ and $\xi \in C_n$.  See \cite{Ca85}, for example, for a discussion of Deligne-Lusztig induction.  This ring structure is studied in \cite{dignemichel}, and it turns out to be exactly the same structure induced by a product on class functions of $U_n$ using Hall polynomials, which was studied by Ennola \cite{ennola}.  In particular, given $\bmu, \bnu, \blam \in \cP^{\tilde\Phi_2}$, we define the Hall polynomial $g_{\bmu \bnu}^{\blam}(t)$ by
\begin{equation} \label{multihall}
g_{\bmu\bnu}^{\blam}(t) = \prod_{s \in \tilde\Phi_2} g_{\bmu^{(s)} \bnu^{(s)}}^{\blam^{(s)}} (t^{|s|}).
\end{equation}
The following result is proven in \cite{TV05}, but also seems to be implicit in \cite{dignemichel}.

\begin{theorem} \label{halldl}
Let $\chi\in C_m$, $\eta\in C_n$, and let $\blam \in \cP^{\tilde\Phi_2}_{m+n}$.  Then the value of $\chi \circ \xi$ on the conjugacy class $\blam$ of $\U(m+n,\FF_{q^2})$ is given by
$$ (\chi \circ \xi)(\blam) = \sum_{\bmu \in \cP^{\tilde\Phi_2}_m, \bnu \in \cP^{\tilde\Phi_2}_n} g_{\bmu \bnu}^{\blam}(-q) \chi(\bmu)\xi(\bnu).$$
\end{theorem}

In the case that $q$ is odd, Henderson \cite{hender} has decomposed the permutation character ${\rm Ind}_{Sp_{2n}}^{U_{2n}}({\bf 1})$ into irreducible constituents, and proved that it is multiplicity free.  It is well known that the Gelfand-Graev character is multiplicity free (see \cite{St67}, for example), and the decomposition into irreducibles for the finite unitary group is given in terms of Deligne-Lusztig characters in \cite{dellusz} and in terms of irreducibles in \cite{ohm}.  Using these results, along with the characteristic map of the finite unitary group, the following result is proven in \cite{TV05}.

\begin{theorem} \label{model}
Let $q$ be odd, and let $U_n = {\rm U}(n, \FF_{q^2})$, $Sp_{2k}=\Sp(2k,\FF_q)$ and let $\mathbf{1}$ be the trivial character.  The sum of all distinct irreducible complex characters of $U_n$ is given by
$$\sum_{\chi \in {\rm Irr}(U_n)} \chi = \sum_{k + 2l = n} \Gamma_{(k)} \circ {\rm Ind}_{Sp_{2l}}^{U_{2l}}({\bf 1}).$$
\end{theorem}

Giving the sum of all distinct irreducible characters of a group as a sum of one-dimensional characters induced from subgroups is called a \emph{model} for the group.  In Theorem \ref{model}, Deligne-Lusztig induction is used in place of subgroup induction, so this is a slight variant from the classical situation.  Note that since Deligne-Lusztig induction does not, in general, produce characters, it is somewhat surprising that all of these products do in fact result in characters.  A model for the group ${\rm GL}(n, \FF_q)$ was first found by Klyachko \cite{kl}, and made more explicit by Inglis and Saxl \cite{inglissaxl}, and the result for that case is analogous to Theorem \ref{model}, except with Deligne-Lusztig induction replaced by parabolic induction.

\subsection{Values of the permutation character} \label{permchar}

Let $U_n=\U(n,\FF_{q^2})$ and $Sp_{2n}=\Sp(2n,\FF_q)$.  
The goal of this section is to use results from Section \ref{conjclasses} to give a closed formula for the value of 
$$\mathbf{1}_{Sp_{2n}}^{U_{2n}}={\rm Ind}_{Sp_{2n}}^{U_{2n}}({\bf 1})$$
evaluated at an arbitrary conjugacy class of $U_{2n}$.  Throughout this section, we let $q$ be the power of an odd prime.

Let $g \in U_{2n}$.  From the formula for an induced character, we have
\begin{equation} \label{induced}
{\bf 1}_{Sp_{2n}}^{U_{2n}}(g) = \frac{ \big|\{ u \in U_{2n} \, \mid \, ugu^{-1} \in Sp_{2n} \} \big|}{|Sp_{2n}|} = \frac{|C_{U_{2n}}(g)|}{|Sp_{2n}|} \big| \{ h \in Sp_{2n} \, \mid \, h \sim g \text{ in } U_{2n} \} \big|.
\end{equation}
If $g$ is in the conjugacy class indexed by $\bnu \in \cP^{\tilde\Phi_2}_{2n}$, then the value of $|C_{U_{2n}}(g)|=a_{\bnu}$, as given by Theorem \ref{wall}.  So, we need to know the number of elements $h$ in $Sp_{2n}$ which are conjugate to $g$ in $U_{2n}$.

Let $s \in\{\pm 1\}$, and let $\bnu \in \cP^{\tilde\Phi_2}$.  Writing $m_i = m_i(\bnu^{(s)})$, define $A(s, \bnu^{(s)}, i)$ and $B(s, \bnu^{(s)})$ as follows.
\begin{align} 
A(s, \bnu^{(s)}, i) &= \prod_{j=1}^{\lfloor m_i/2 \rfloor} (1 - 1/q^{2j})=q^{-2(\lfloor m_i/2\rfloor)^2 - \lfloor m_i/2 \rfloor} |\Sp(2\lfloor m_i/2\rfloor,\FF_q)|,\notag\\
B(s, \bnu^{(s)}) &= q^{|\bnu^{(s)}|/2 + n(\bnu^{(s)}) + o(\bnu^{(s)})/2} \prod_i A(s, \bnu^{(s)}, i). \label{ABuni}
\end{align}

\begin{theorem} \label{permcharvals}
Let $\bnu \in \cP^{\tilde\Phi_2}_{2n}$, and let $q$ be odd.  Then ${\bf 1}_{Sp_{2n}}^{U_{2n}}(\bnu) = 0$ unless 

\noindent $\mathrm{(i)}$ For every $s \in \tilde\Phi_2$, $\bnu^{(s)} = \bnu^{(s^*)}$.

\noindent $\mathrm{(ii)}$ For every odd $j\geq 1$, $m_j(\bnu^{(1)}),m_j(\bnu^{(-1)})\in 2\ZZ_{\geq 0}$.

\noindent
If both $\mathrm{(i)}$ and $\mathrm{(ii)}$ are satisfied, then
$${\bf 1}_{Sp_{2n}}^{U_{2n}}(\bnu) = \frac{a_{\bnu}}{\prod_{s \in \tilde\Phi_2} B(s, \bnu^{(s)})},$$
where $a_{\bnu}$ is as defined in Theorem \ref{wall}, and $B(s, \bnu^{(s)})$ is defined as in (\ref{ABdefn}) and (\ref{ABuni}).
\end{theorem}
\begin{proof}  Lemma \ref{UpDownConjugacy} (d) and (\ref{induced}) imply that ${\bf 1}_{Sp_{2n}}^{U_{2n}}(\bnu)=0$ if $\bnu$ does not satisfy (i) and (ii) (else the conjugacy class parameterized by $\bnu$ is not conjugate to an element in $Sp_{2n}$).  

Let $h \in Sp_{2n}$, and suppose $h$ is in the $U_{2n}$-conjugacy class parameterized by $\bnu \in \cP^{\Phi}_{2n}$.  Thus, by Lemma \ref{UpDownConjugacy} (d), if $h$ is in the $Sp_{2n}$-conjugacy class parameterized by $\bmu\in \cP^{\pm\Phi^*}_{2n}$, then 
$$\bmu^{(1)^\circ}=\bnu^{(1)},\quad \bmu^{(-1)^\circ}=\bnu^{(-1)}, \quad \text{and}\quad \bmu^{(s\cup s^*)}=\bnu^{(s)}, \quad\text{for all $s\cup s^*\in \Phi^*\setminus\{\pm 1\}$.}$$
From (\ref{induced}) and Theorem \ref{spucent}, we have
\begin{align*} 
{\bf 1}_{Sp_{2n}}^{U_{2n}}(\bnu) &= \frac{a_{\bnu}}{|Sp_{2n}|}\sum_{ \mu,\gamma \in \cP^{\pm} \atop \mu^{\circ} = \bnu^{(1)},\gamma^\circ=\bnu^{(-1)}} \frac{|Sp_{2n}|}{B^*(1, \mu) B^*(-1, \gamma) \prod_{ s\in \tilde\Phi_2\setminus\{\pm 1\}} B(s, \bnu^{(s)})}\\
&= \frac{a_{\bnu}}{\prod_{ s \in \tilde\Phi_2\setminus\{\pm 1\}} B(s, \bnu^{(s)})}  \sum_{ \mu,\gamma \in \cP^{\pm} \atop \mu^{\circ} = \bnu^{(1)},\gamma^\circ=\bnu^{(-1)}}  \frac{1}{B^*(1, \mu) B^*(-1, \gamma)}.
\end{align*}
So we want to show
\begin{equation} \label{reduce1}
\frac{1}{B(1, \bnu^{(1)}) B(-1, \bnu^{(-1)})} =  \sum_{ \mu,\gamma \in \cP^{\pm} \atop \mu^{\circ} = \bnu^{(1)},\gamma^\circ=\bnu^{(-1)}}  \frac{1}{B^*(1, \mu) B^*(-1, \gamma)}.
\end{equation}
Note that we can think of the set of symplectic signed partitions $\cP^\pm$ as
$$\cP^\pm=\{(\nu,\delta)\in \cP\times \cS_\nu\},\quad \text{where}\quad \cS_\nu=\left\{\begin{array}{ccc}\delta:\{i\in 2\ZZ_{\geq 1}\ \mid\ \nu_j = i \text{ for some } j \} & \longrightarrow &\{\pm 1\}\\ i & \mapsto & \delta_i\end{array}\right\},$$
so that if $\mu=\bnu^{(1)}$ and $\gamma=\bnu^{(-1)}$, then equation (\ref{reduce1}) becomes
\begin{equation*} 
\frac{1}{B(1, \mu) B(-1, \gamma)} = \sum_{\delta\in \cS_\mu\atop \tau\in \cS_\gamma} \frac{1}{B^*(1, (\mu,\delta)) B^*(-1, (\gamma,\tau))}.
\end{equation*}

From (\ref{Atildedefn}) and (\ref{Btildedefn}), we have for $s\in \{\pm 1\}$ and $(\mu,\delta)\in \cP^\pm$,
\begin{align*}
B^*(s, (\mu,\delta)) &=  q^{|\mu|/2 + n(\mu) - \sum_i m_i^2/2 } \prod_i A^*(s, (\mu,\delta), i)\\ 
&= q^{ |\mu|/2 + n(\mu) - \sum_i m_i^2/2 } \hspace{-.1cm} \prod_{i \text{ odd}} \big|{\rm Sp}(m_i, \FF_q)\big| \prod_{i \text{ even}} q^{m_i/2} \big|{\rm O}^{\delta_i}(m_i, \FF_q)\big|,
\end{align*}
where $m_i = m_i(\mu)$. For $i$ odd,
$$\big|{\rm Sp}(m_i, \FF_q)\big| = q^{m_i^2/4} \prod_{j = 1}^{m_i/2} (q^{2j} - 1) = q^{m_i/2 + m_i^2/2} \prod_{j=1}^{m_i/2} (1 - 1/q^{2j}).$$
Note that
$$ \prod_{i \text{ odd}} q^{m_i/2} = q^{o(\mu)/2}. $$
From (\ref{reduce1}) and (\ref{ABuni}), it suffices to show
\begin{equation} \label{reduce2}
\frac{q^{-\sum_{i \text{ even}} m_i(\mu)^2/2 -\sum_{i \text{ even}} m_i(\gamma)^2/2}}{\prod_{i \text{ even}} A(1, \mu, i) A(-1, \gamma, i)} =  \sum_{\delta\in \cS_\mu\atop \tau\in \cS_\gamma}  \prod_{i \text{ even}}\frac{1}{A^*(1, (\mu,\delta), i) A^*(-1, (\gamma,\tau), i)}. 
\end{equation}
Let $s\in \{\pm 1\}$ and $(\mu,\delta)\in \cP^\pm$.  For $i$ even,  and $m_i=m_i(\mu)$ odd,
$$A^*(s, (\mu,\delta), i)= q^{m_i/2} \big|{\rm O}(m_i, \FF_q) \big| = 2q^{m_i^2/4 + 1/4}\hspace{-.12cm}\prod_{j=1}^{\lfloor m_i/2 \rfloor}\hspace{-.1cm} (q^{2j} - 1) = 2q^{m_i^2/2}\hspace{-.12cm}\prod_{j=1}^{\lfloor m_i/2 \rfloor}\hspace{-.1cm} (1 - 1/q^{2j}).$$
For $i$ even and $m_i$ even,
\begin{align*}
A^*(s,(\mu,\delta), i) &= q^{m_i/2} \big|{\rm O}^{\delta_i}(m_i, \FF_q) \big|\\
&= 2q^{(m_i)^2/4}(q^{m_i/2} + \delta_i) \prod_{j=1}^{m_i/2 - 1} (q^{2j} - 1)\\
&= 2q^{(m_i)^2/2}(1 +\delta_i q^{-m_i/2})\prod_{j=1}^{m_i/2 - 1} (1 - 1/q^{2j}).
\end{align*}

Returning to (\ref{reduce2}), let $M= \big|\{ i \text{ even} \; | \; m_i(\mu) \neq 0 \} \big|+ \big|\{ i \text{ even} \; | \; m_i(\gamma) \neq 0 \} \big|$.  Then
\begin{align} \label{reduce3}
 \sum_{\delta\in \cS_\mu\atop \tau\in \cS_\gamma} &  \prod_{i \text{ even}}\frac{1}{A^*(1, (\mu,\delta), i) A^*(-1, (\gamma,\tau), i)} \notag \\
& = \frac{q^{ -\sum_{i \text{ even}} m_i(\mu)^2/2-\sum_{i \text{ even}} m_i(\gamma)^2/2}2^{-M}}{\prod_{i \text{ even}} A(1, \mu, i) A(-1, \gamma, i)}  \prod_{i\text{ even}\atop m_i(\mu) \text{ even}} (1 - 1/q^{|m_i(\mu)|})\prod_{i\text{ even}\atop m_i(\mu) \text{ even}}(1 - 1/q^{|m_i(\gamma)|}) \notag \\ 
& \hspace*{2cm} \cdot \sum_{\delta\in \cS_\mu\atop \tau\in \cS_\gamma} \prod_{i\text{ even}\atop m_i(\mu) \text{ even}} \frac{1}{(1 + \delta_i q^{-|m_i(\mu)|/2})} \prod_{i\text{ even}\atop m_i(\mu) \text{ even}} \frac{1}{(1 + \tau_i q^{-|m_i(\gamma)|/2})}.
\end{align}
In the last sum there is a choice of signs for $\delta_i$ and $\tau_i$ whenever $i$ is even and the multiplicities are nonzero, and so there are $2^{M}$ terms in the sum.  Define 
\begin{align*}
 M_e &= \big| \{i \text{ even} \; | \; m_i(\mu) \neq 0 \text{ and } m_i(\mu) \text{ is even} \}\big|+\big| \{i \text{ even} \; | \; m_i(\gamma) \neq 0 \text{ and } m_i(\gamma) \text{ is even} \}\big|\\
M_o &= \big| \{ i \text{ even} \; | \; m_i(\mu) \text{ is odd}\}\big|+ \big| \{ i \text{ even} \; | \; m_i(\mu) \text{ is odd}\}\big|, 
 \end{align*}
so that $M = M_e + M_o$.  Since the choices of sign for $m_i$ odd do not affect the sum in (\ref{reduce3}), we have
\begin{align} \label{reduce4}
 \sum_{\delta\in \cS_\mu\atop \tau\in \cS_\gamma} \prod_{i\text{ even}\atop m_i(\mu)\text{ even}}& \frac{1}{(1 + \delta_i q^{-m_i(\mu)/2})} \prod_{i\text{ even}\atop m_i(\gamma) \text{ even}} \frac{1}{(1 + \tau_i q^{-m_i(\gamma)/2})} \notag \\
&= \sum_{\delta_i,\tau_j\in \{\pm 1\} \atop{ i, m_i(\mu),j,m_j(\gamma) \text{ even}}} \hspace{-.5cm} 2^{M_o}\hspace{-.25cm} \prod_{i\text{ even}\atop m_i(\mu) \text{ even}} \frac{1}{1 + \delta_i/q^{m_i(\mu)/2}} \prod_{j\text{ even}\atop m_j(\gamma)\text{ even}}\frac{1}{1 + \tau_j/q^{m_j(\gamma)/2}}, 
\end{align}
where there are exactly $2^{M_e}$ terms in the second sum.

For any set of numbers $\{\alpha_1, \ldots, \alpha_k \}$, an induction argument gives us the identity
$$ \sum_{\delta_1,\delta_2,\ldots, \delta_k \in\{\pm 1\}} \prod_{j = 1}^k \frac{1}{1 + \delta_j \alpha_j} = \frac{2^k}{\prod_{j=1}^k (1 - \alpha_j^2)}.$$
Applying this identity to (\ref{reduce4}), we obtain
\begin{align} \label{reduce5}
 \sum_{\delta_i,\tau_j\in \{\pm 1\} \atop{ i, m_i(\mu),j,m_j(\gamma) \text{ even}}} \hspace{-.5cm} 2^{M_o}\hspace{-.25cm} \prod_{i\text{ even}\atop m_i(\mu) \text{ even}} & \frac{1}{1 + \delta_i/q^{m_i(\mu)/2}} \prod_{j\text{ even}\atop m_j(\gamma)\text{ even}}\frac{1}{1 + \tau_j/q^{m_j(\gamma)/2}}\notag\\
&= \frac{2^{M}}{\prod_{i, m_i(\mu) \text{ even}} (1 - 1/q^{|m_i(\mu)|})\prod_{i, m_i(\gamma) \text{ even}}(1 - 1/q^{|m_i(\gamma)|}) }.
\end{align}
Substituting (\ref{reduce5}) into (\ref{reduce3}), we obtain the desired result (\ref{reduce2}).
 \end{proof}

Theorem \ref{permcharvals} has an especially nice form when evaluated at unipotent elements.  

\begin{corollary} \label{spunipotent}
Let $u_{\mu} \in U_n$ be a unipotent element of type $\mu$.  Unless $m_i(\mu) \in 2\ZZ_{\geq 0}$ whenever $i$ is odd, ${\bf 1}_{Sp_{2n}}^{U_{2n}}(u_{\mu}) = 0$, and otherwise we have
$${\bf 1}_{Sp_{2n}}^{U_{2n}}(u_{\mu}) = \frac{(-1)^{|\mu|} a_{\mu}(-q)}{B(1, \mu)} = q^{n(\mu) + (|\mu| - o(\mu))/2} c_{\mu}(-1/q).$$
\end{corollary}

Corollary \ref{spunipotent} immediately follows from Theorem \ref{permcharvals} and the following result.

\begin{lemma} \label{cfactor}
Let $\lambda \in \cP$.  Then
$$\frac{a_{\lambda}(-q)}{B(1, \lambda)} = (-1)^{|\lambda|} q^{n(\lambda) + (|\lambda| - o(\lambda))/2}c_{\lambda}(-1/q).$$
\end{lemma}
\begin{proof} Recall the definition of $c_{\lambda}(t)$ as given in (\ref{cdefn}).  From the form of this function given in (\ref{cprod}), we have
$$ c_{\lambda}(-1/q) = \frac{\prod_i \prod_{j = 1}^{m_i(\lambda)} (1 - (-1/q)^j)}{\prod_i \prod_{j = 1}^{\lfloor m_i(\lambda)/2 \rfloor} (1 - (-1/q)^{2j})}.$$ 
The definitions of $a_{\lambda}(-q)$ in Theorem \ref{wall} and of $B(1, \lambda)$ in (\ref{ABuni}) now imply the result.
\end{proof}

\subsection{Values of the full character sum} \label{fullsum}

We now calculate the sum of all characters of $U_n$ at an arbitrary element.

\begin{theorem} \label{charvalues}
Let $q$ be odd, and let $\blam \in \cP^{\tPhi_2}_n$.  Then $\sum_{\chi \in {\rm Irr}(U_n)} \chi(\blam) = 0$ unless

\noindent $\mathrm{(i)}$ For every $s \in \tPhi_2$, $\blam^{(s)} = \blam^{(s^*)}$.

\noindent $\mathrm{(ii)}$ For every odd $i$, $m_i(\blam^{(-1)}) \in 2\ZZ_{\geq 0}$ .

\noindent $\mathrm{(iii)}$ For every even $i$, $m_i(\blam^{(1)}) \in 2\ZZ_{\geq 0}$.

\noindent
If $\mathrm{(i)}$, $\mathrm{(ii)}$, and $\mathrm{(iii)}$ are satisfied, then
$$\sum_{\chi \in {\rm Irr}(U_n)} \chi(\blam) = \frac{a_{\blam} \, q^{o(\blam^{(1)})}}{\prod_{s \in \tPhi_2} B(s, \blam^{(s)})},$$
where $a_{\blam}$ is as defined in Theorem \ref{wall}, and $B(s, \blam^{(s)})$ is defined as in (\ref{ABdefn}) and (\ref{ABuni}).
\end{theorem}
\begin{proof}
Let $\blam \in \cP^{\tPhi_2}_n$.  From Theorems \ref{halldl} and \ref{model}, we have
\begin{align*}
\sum_{\chi \in {\rm Irr}(U_n)} \chi(\blam) &= \sum_{k + 2l = n} \Big(\Gamma_{(k)} \circ {\rm Ind}_{Sp_{2l}}^{U_{2l}} ({\bf 1})\Big) (\blam)\\
&= \sum_{k+2l = n} \sum_{ \bmu \in \cP^{\tPhi_2}_k \atop {\bnu \in \cP^{\tPhi_2}_{2l}}} g_{\bmu \bnu}^{\blam}(-q) \Gamma_{(k)}(\bmu) {\bf 1}_{Sp_{2l}}^{U_{2l}}(\bnu)
\end{align*}
From Equation (\ref{UGGFormula}), we have $\Gamma_{(k)}(\bmu) = 0$ unless $\bmu$ is unipotent, so we may assume $\mu = \bmu^{(1)}$, where $|\mu| = k$, and $\bmu^{(s)} = \emptyset$ when $s \neq 1$.  As in (\ref{multihall}), we have,
$$ g_{\bmu \bnu}^{\blam} (-q) = \prod_{s \in \tPhi_2} g_{\bmu^{(s)} \bnu^{(s)}}^{\blam^{(s)}} ((-q)^{|s|}).$$
Since $\bmu^{(s)} = \emptyset$ when $s \neq 1$, then from (\ref{hallempty}) we have that if $s \neq 1$,
$$ g_{\bmu^{(s)} \bnu^{(s)}}^{\blam^{(s)}} ((-q)^{|s|}) = \begin{cases} 1 & \text{ if } \bnu^{(s)} = \blam^{(s)}, \\ 0 & \text{ otherwise.} \end{cases}$$
From these facts, and from the value of $\Gamma_{(k)}$ given in(\ref{UGGFormula}), we have
\begin{align*}
\sum_{\chi \in {\rm Irr}(U_n)} \chi(\blam) = \sum_{k + 2l = n} \sum_{\mu \in \cP_k} \sum_{\bnu \in \cP^{\tPhi_2}_{2l}, \bnu^{(1)} = \nu \atop { \bnu^{(s)} = \blam^{(s)}, s \neq 1}} \hspace{-.30cm} g_{\mu \nu}^{\blam^{(1)}} (-q) (-1)^{\lfloor k/2 \rfloor - \ell(\mu)} \prod_{i=1}^{\ell(\mu)} \bigl( (-q)^i -1 \bigr) {\bf 1}_{Sp_{2l}}^{U_{2l}}(\bnu).
\end{align*}
Since $|\mu| + |\nu| = |\blam^{(1)}|$ and $|\nu|$ is even in the above sum, we have
$$(-1)^{\lfloor |\blam^{(1)}|/2 \rfloor} = (-1)^{\lfloor |\mu|/2 \rfloor} (-1)^{|\nu|/2} \;\; \text{ and } (-1)^{|\blam^{(1)}|} = (-1)^{|\mu|},$$
and so we multiply the outside of the sum by the expressions on the left sides of the equations above, and the inside of the sum by the expressions on the right sides.  We have $\bnu^{(s)} = \blam^{(s)}$ when $s \neq 1$, and from Theorem \ref{permcharvals}, ${\bf 1}_{Sp_{2l}}^{U_{2l}}(\bnu)=0$ unless $\blam^{(s)} = \blam^{(s^*)}$ for $s \in \tPhi_2$, $m_i(\blam^{(-1)})$ is even when $i$ is odd, and $m_i(\bnu^{(1)})$ is even when $i$ is odd, and otherwise we have
\begin{align*}
{\bf 1}_{Sp_{2l}}^{U_{2l}}(\bnu) = \frac{a_{\bnu}}{\prod_{s \in \tPhi_2} B(s, \bnu^{(s)})}
= \frac{\prod_{s \neq 1} (-1)^{|\blam^{(s)}|\,|s|}\, a_{\blam^{(s)}}\big((-q)^{|s|}\big)}{\prod_{s \neq 1} B(s, \blam^{(s)})} \cdot \frac{a_{\nu}(-q)}{B(1, \nu)},
\end{align*}
where $\nu = \bnu^{(1)}$.  So, $\sum_{\chi \in {\rm Irr}(U_n)} \chi(\blam) = 0$ unless $\blam^{(s)} = \blam^{(s^*)}$ for $s \in \Phi$ and $m_i(\blam^{(-1)})$ is even when $i$ is odd, and otherwise we have
\begin{align*}
& \sum_{\chi \in {\rm Irr}(U_n)} \chi(\blam) = (-1)^{\lfloor |\blam^{(1)}|/2 \rfloor + |\blam^{(1)}|} \prod_{s \neq 1} (-1)^{|\blam^{(s)}|\,|s|}\, a_{\blam^{(s)}}\big((-q)^{|s|}\big) \\
&\hspace*{1.5cm} \cdot \sum_{\mu, \nu \in \cP \atop{m_k(\nu) \text{ even} \atop{\text{for $k$ odd}}}}\hspace{-.13cm}  g_{\mu \nu}^{\blam^{(1)}} (-q) (-1)^{|\mu|} \prod_{i=1}^{\ell(\mu)} \bigl( 1 - (-q)^i \bigr) \frac{(-1)^{|\nu|/2} a_{\nu}(-q)}{B(1, \nu)\prod_{s \neq 1} B(s, \blam^{(s)})}.
\end{align*}
From Lemma \ref{HLprod}, with $t = -1/q$, this expression is the coefficient of $P_{\blam^{(1)}}(x;-1/q)$ in the expansion of  
\begin{align*}
&(-q)^{n(\blam^{(1)})} (-1)^{\lfloor |\blam^{(1)}|/2 \rfloor + |\blam^{(1)}|} \prod_{s \neq 1} (-1)^{|\blam^{(s)}|\,|s|}\, a_{\blam^{(s)}}\big((-q)^{|s|}\big) \\
& \cdot \sum_{\mu \in \cP} \frac{P_{\mu}(x; -1/q)}{(-q)^{n(\mu)}} (-1)^{|\mu|} \prod_{i=1}^{\ell(\mu)} \bigl(1- (-q)^i \bigr) \hspace{-.3cm} \sum_{\nu \in \cP \atop{m_k(\nu) \text{ even} \atop{\text{for $k$ odd}}}}\frac{P_{\nu}(x; -1/q)}{(-q)^{n(\nu)}} \frac{(-1)^{|\nu|/2} a_{\nu}(-q)}{B(1, \nu)\prod_{s \neq 1} B(s, \blam^{(s)})}.
\end{align*}
By applying Lemma \ref{HLsum1}, with $t = -1/q$ and $y = q$, and replacing each $x_i$ with $-x_i$, this expression is the coefficient of $P_{\blam^{(1)}}(x;-1/q)$ in the expansion of
\begin{align*}
&(-q)^{n(\blam^{(1)})} (-1)^{\lfloor |\blam^{(1)}|/2 \rfloor + |\blam^{(1)}|} \prod_{s \neq 1} (-1)^{|\blam^{(s)}|\,|s|}\,  a_{\blam^{(s)}}\big((-q)^{|s|}\big) \, \prod_{i \geq 1} \frac{1 - x_i q}{1 + x_i}\\
&\hspace*{1.5cm} \cdot \sum_{\nu \in \cP \atop{m_k(\nu) \text{ even} \atop{\text{for $k$ odd}}}}\frac{P_{\nu}(x; -1/q)}{(-q)^{n(\nu)}} \frac{(-1)^{|\nu|/2} a_{\nu}(-q)}{B(1, \nu)\prod_{s \neq 1} B(s, \blam^{(f)})}.
\end{align*}
Applying Lemmas \ref{cfactor} and \ref{evensign}, this expression simplifies to
\begin{align*}
&(-q)^{n(\blam^{(1)})} (-1)^{\lfloor |\blam^{(1)}|/2 \rfloor + |\blam^{(1)}|} \prod_{s \neq 1} (-1)^{|\blam^{(s)}|\,|s|}\,  a_{\blam^{(s)}}\big((-q)^{|s|}\big) \, \prod_{i \geq 1} \frac{1 - x_i q}{1 + x_i}\\
&\hspace*{1.5cm} \cdot \sum_{\nu \in \cP \atop{m_k(\nu) \text{ even} \atop{\text{for $k$ odd}}}}P_{\nu}(x; -1/q) \frac{(-q)^{(|\nu| - o(\nu))/2} c_{\nu}(-1/q)}{\prod_{s \neq 1} B(s, \blam^{(s)})}.
\end{align*}
We may now apply Theorem \ref{HLKa}, with $t = -1/q$, to change the last sum into a product, which simplifies the above expression to
\begin{align*}
&(-q)^{n(\blam^{(1)})} (-1)^{\lfloor |\blam^{(1)}|/2 \rfloor + |\blam^{(1)}|} \prod_{s \neq 1} (-1)^{|\blam^{(s)}|\,|s|}\,  a_{\blam^{(s)}}\big((-q)^{|s|}\big) \\
&\hspace*{1.5cm} \cdot \prod_{i \geq 1} \frac{1 - x_i q}{1 + x_i} \prod_{i \leq j} \frac{1-x_i x_j}{1+x_i x_j q} \frac{1}{\prod_{f \neq 1} B(s, \blam^{(s)})}.\end{align*}
Finally, Theorem \ref{HLFG} with $t = -1/q$ implies that the coefficient of $P_{\blam^{(1)}}(x; -1/q)$ of the above expression is $0$ unless $m_k(\blam^{(1)})$ is even whenever $k$ is even, in which case it is
$$(-q)^{n(\blam^{(1)})} (-1)^{\lfloor |\blam^{(1)}|/2 \rfloor + |\blam^{(1)}|} \prod_{s \neq 1} (-1)^{|\blam^{(s)}|\,|s|}\,  a_{\blam^{(s)}}\big((-q)^{|s|}\big)\frac{c_{\blam^{(1)}}(-1/q) (-q)^{ (o(\blam^{(1)}) + |\blam^{(1)}|)/2}}{\prod_{s \neq 1} B(s, \blam^{(s)})}.$$
Applying Lemmas \ref{cfactor} and \ref{evensign} and simplifying gives the desired result.
\end{proof}

In the case that we are finding the sum of the characters of $U_n$ at a unipotent element, we obtain a result much like Corollary \ref{spunipotent}.

\begin{corollary} \label{unipotent}
Let $u_{\mu} \in U_n$ be a unipotent element of type $\mu$.  Unless $m_i(\mu) \in 2\ZZ_{\geq 0}$ whenever $i$ is even, $\sum_{\chi \in {\rm Irr}(U_n)} \chi(u_{\mu}) = 0$, and otherwise we have
$$\sum_{\chi \in {\rm Irr}(U_n)} \chi(u_{\mu}) = \frac{(-1)^{|\mu|} a_{\mu}(-q) q^{o(\mu)}}{B(1, \mu)} = q^{n(\mu) + (|\mu| + o(\mu))/2} c_{\mu}(-1/q).$$
\end{corollary}
\begin{proof} This follows immediately from Theorem \ref{charvalues} and Lemma \ref{cfactor}.
\end{proof}

\section{Probabilities and Frobenius-Schur indicators} \label{probs}

Here we reinterpret the main results in Sections \ref{permchar} and \ref{fullsum} as probabilistic statements, as the corresponding results for ${\rm GL}(n, \FF_q)$ are given in \cite{fulgural}.  For the result in Theorem \ref{permcharvals}, we have the following statement.

\begin{corollary} \label{spprop} 
Let $q$ be odd.  The probability that a uniformly randomly selected element of ${\rm Sp}(2n, \FF_q)$ belongs to the conjugacy class in ${\rm U}(2n, \FF_{q^2})$ corresponding to $\bnu \in \cP^{\tPhi_2}_{2n}$ is $0$ unless $\bnu^{(s)} = \bnu^{(s^*)}$ for every $s \in \tPhi_2$, and $m_j(\bnu^{(1)})$ and $m_j(\bnu^{(-1)})$ are even for every odd $j \geq 1$, in which case it is equal to
$$ \frac{1}{\prod_{s \in \tPhi_2} B(s, \bnu^{(s)})}.$$
\end{corollary}
\begin{proof}
From Equation (\ref{induced}), we have, for any $g \in U_{2n}$,
$$ {\bf 1}_{Sp_{2n}}^{U_{2n}}(g) = \frac{|C_{U_{2n}}(g)|}{|Sp_{2n}|} \big| \{ h \in Sp_{2n} \, \mid \, h \sim g \text{ in } U_{2n} \} \big|.$$
Taking $g$ to be in the conjugacy class corresponding to $\bnu \in \cP^{\tPhi_2}_{2n}$, we have $|C_{U_{2n}}(g)| = a_{\bnu}$ from Theorem \ref{wall}.  So, the probability we want is exactly
$$ \frac{{\bf 1}_{Sp_{2n}}^{U_{2n}}(\bnu)}{a_{\bnu}} = \frac{\big| \{ h \in Sp_{2n} \, \mid \, h \sim g \text{ in } U_{2n} \} \big|}{|Sp_{2n}|}. $$
The result now follows directly from Theorem \ref{permcharvals}. \end{proof}

Interpreting the result for the full character sum in Theorem \ref{charvalues} as a probabilistic statement is most easily accomplished using the twisted Frobenius-Schur indicator, originally defined in \cite{kawmat} and further studied in \cite{bumpginz}.  Let $G$ be a finite group, with automorphism $\iota$, which either has order 2 or is the identity.  Let $(\pi, V)$ be an irreducible complex representation of $G$, and suppose that ${^\iota \pi} \cong \hat{\pi}$, where $\hat{\pi}$ is the contragredient representation of $\pi$ and ${^\iota \pi}$ is defined by ${^\iota \pi}(g) = \pi({^\iota g})$.  The equivalence of the representations ${^\iota \pi}$ and $\hat{\pi}$ implies that there exists a nondegenerate bilinear form
$$ \langle \cdot, \cdot \rangle: V \times V \rightarrow \CC, $$
unique up to scalar by Schur's Lemma, such that
$$ \langle \pi(g)v, {^\iota \pi}(g)w \rangle = \langle v, w \rangle,$$
for every $g \in G$, $v, w \in V$.

Because the bilinear form is unique up to scalar, we have
$$\langle v, w \rangle = \varepsilon_{\iota}(\pi) \langle w, v \rangle,$$
where $\varepsilon_{\iota}(\pi) = \pm 1$ depends only on $\iota$ and $\pi$.  If ${^\iota \pi} \not\cong \hat{\pi}$, we define $\epi(\pi) = 0$.  Then $\epi(\pi)$ is the {\em twisted Frobenius-Schur indicator} of $(\pi, V)$ with respect to $\iota$.  If $\chi$ is the character corresponding to $\pi$, we also write $\epi(\pi) = \epi(\chi)$.  The following results are proven in \cite{kawmat} and \cite{bumpginz}, and show that the twisted Frobenius-Schur indicator indeed generalizes the classical Frobenius-Schur indicator, which is the case that $\iota$ is trivial.

\begin{theorem} \label{fsprops}
Let $G$ be a finite group with automorphism $\iota$ such that $\iota^2$ is the identity.  Then we have the formulas 
$$\epi(\chi) = \frac{1}{|G|} \sum_{g \in G} \chi(g\; {^\iota g}), \;\; \text{ and } \;\; \sum_{\chi \in {\rm Irr}(G)} \epi(\chi)\chi(g) = \big|\{ h \in G \; | \; h\;{^\iota h} = g \} \big|.$$
\end{theorem}

Now consider the case $G = {\rm U}(n, \FF_{q^2})$ and $\iota$ the transpose inverse automorphism.  In \cite{TV05}, it is proven that $\epi(\pi) = 1$ for every irreducible complex representation $(\pi, V)$ of ${\rm U}(n, \FF_{q^2})$.  By Theorem \ref{fsprops}, this is equivalent to \cite[Corollary 5.2]{TV05}
\begin{equation} \label{degreesum}
\sum_{\chi \in {\rm Irr}({\rm U}(n, \FF_{q^2}))} \chi(1) = \big| \{ g \in {\rm U}(n, \FF_{q^2}) \; | \; g \text{ symmetric} \} \big|.
\end{equation}
We conclude with the following probabilistic version of Theorem \ref{charvalues}.

\begin{corollary} \label{fullprob}
Let $q$ be odd.  Let $u$ be a uniformly randomly selected element of ${\rm U}(n, \FF_{q^2})$, and let $\iota$ be the transpose inverse automorphism of ${\rm U}(n, \FF_{q^2})$.  The probability that $u \, {^\iota u}$ is in the conjugacy class corresponding to $\bnu \in \cP^{\tPhi_2}_n$ is $0$ unless $\bnu^{(s)} = \bnu^{(s^*)}$ for every $s \in \tPhi_2$, $m_j(\bnu^{(1)})$ is even for every even $j$,  and $m_j(\bnu^{(-1)})$ is even for every odd $j$, in which case it is equal to 
$$\frac{q^{o(\blam^{(1)})}}{\prod_{s \in \tPhi_2} B(s, \blam^{(s)})}.$$
\end{corollary}
\begin{proof}  Let $g \in U_{2n}$ be an element from the conjugacy class corresponding to $\bnu$.  From Theorem \ref{fsprops}, and the fact that $\epi(\pi) = 1$ for every irreducible representation $\pi$ of $U_n$, we have
$$ \sum_{\chi \in {\rm Irr}(U_n)} \chi(\bnu) = \big| \{ u \in U_n \; | \; u\, {^\iota u} = g \} \big|.$$
We need to count each set on the right as $g$ ranges over the conjugacy class corresponding to $\bnu$, which has size $|U_{2n}|/a_{\bnu}$ by Theorem \ref{wall}.  Multiplying by this quantity, and dividing by $|U_{2n}|$ to get a probability, we find that the desired probability is
$$ \frac{\big| \{ u \in U_n \; | \; u\, {^\iota u} = g \} \big|}{a_{\bnu}}= \frac{\sum_{\chi \in {\rm Irr}(U_n)} \chi(\bnu)}{a_{\bnu}}.$$
The result now follows from Theorem \ref{charvalues}.
\end{proof}

\end{document}